\documentclass[11pt, reqno]{amsart}

\usepackage{epigraph}
\usepackage{enumerate}
\usepackage{amssymb,epsfig}
\usepackage{multicol, graphicx}
\usepackage{hyperref}
\usepackage{xcolor}
\usepackage[all]{xypic}

\usepackage{color}
\usepackage{comment}

\theoremstyle{plain}
\newtheorem{THM}{Theorem}

\newtheorem{thm}{Theorem}[section]
\newtheorem{pro}[thm]{Proposition}

\newtheorem{lem}[thm]{Lemma}

\theoremstyle{definition}

\theoremstyle{remark}
\newtheorem{rem}[thm]{Remark}

\newtheorem{exa}[thm]{Example}

\newcommand{\F}{{\mathcal F}}
\newcommand{\calf}{{\mathcal F}}
\newcommand{\calF}{{\mathcal F}}
\newcommand{\calfp}{\overline{{\mathcal F}}}

\newcommand{\calg}{{\mathcal G}}

\newcommand{\calo}{{\mathcal O}}
\newcommand{\calO}{{\mathcal O}}

\newcommand{\calR}{{\mathcal R}}

\newcommand{\Hhyp}{{\mathbb H}}

\newcommand{\tang}{{\rm tang}}

\renewcommand{\P}{{\mathbb P}}
\newcommand{\Pd}{{\mathbb P}^2}
\newcommand{\Pu}{{\mathbb P}^1}

\newcommand{\C}{{\mathbb C}}
\newcommand{\D}{{\mathbb D}}

\newcommand{\N}{{\mathbb N}}
\newcommand{\R}{{\mathbb R}}
\newcommand{\Q}{{\mathbb Q}}
\newcommand{\Z}{{\mathbb Z}}
\newcommand{\G}{{\mathcal G}}

\newcommand{\tor}{\xymatrix{\ar@{-->}[r]&}}
 
\newcommand{\PSL}{\mathrm{PSL}}
\DeclareMathOperator{\Aut}{\mathrm{Aut}}

\DeclareMathOperator{\sing}{\mathrm{Sing}}
\DeclareMathOperator{\rk}{\mathrm{rk}}

\newcommand{\equivq}{=_{\Q}}
\newcommand{\der}{\partial}

\AtBeginDocument{
   \def\MR#1{}
}
\begin{document}

\begin{abstract}
We propose a study of the foliations of the projective plane induced by simple derivations of the polynomial ring in two 
indeterminates over the complex field.
These correspond to foliations which have no invariant algebraic curve 
nor singularities in the complement of a  line. We establish  the position of these foliations in 
the birational classification of foliations and  prove the finiteness of their   birational symmetries. 
Most  of the results apply to wider classes of foliations.
\end{abstract}

\author[G. Cousin]{Ga\"el Cousin}
\address{GMA-IME-UFF\\
Campus do Gragoat\'a\\  Niter\'oi \\ RJ \\Brazil}
\email{gcousin@id.uff.br }
\author[L.G. Mendes]{
Lu\'is Gustavo Mendes}
\address{Instituto de Matem\'atica\\ UFRGS\\ Porto Alegre\\Brazil }
\email{mendes@mat.ufrgs.br}
\author[I. Pan]{Iv\'an Pan}
\address{Centro de Matem\'atica\\ Facultad de Ciencias\\ UdelaR\\ Montevideo\\Uruguay}
\email{ivan@cmat.edu.uy}
\title[Foliations and Simple Derivations]{Birational  geometry  of foliations associated to simple derivations}
\thanks{G. Cousin was partially funded by FIRB RBFR12W1AQ, Labex IRMIA and ANR-13-JS01-0002-01.  I. Pan was partially supported by ANII and PEDECIBA of Uruguay}

\maketitle


\section{Introduction and  results}

\epigraph{"(...) mais nous ne serons satisfaits que quand on aura trouv\'e un certain groupe de transformations (par exemple de transformations de Cremona) qui jouera, par rapport aux \'equations diff\'erentielles, le m\^ eme r\^ole que le groupe des transformations birationnelles pour les courbes alg\'ebriques. Nous pourrons alors ranger dans une m\^eme classe toutes les transform\'ees d'une m\^eme equation."}{--- \textup{Henri Poincar\'e}, L'Avenir des Math\'ematiques}

At least since Poincar\'e \cite{hp1891crd}, the study of the algebraic subsets that are left invariant by a given plane polynomial vector field is known to be a difficult matter. For instance, the example~${rx\partial_x+y\partial_y, r\in \Q}$ shows that it is in general impossible to bound the degree of the invariant algebraic curves in terms of the degrees of the vector field's coefficients.
In commutative algebra, the vector fields  that preserve no nontrivial  algebraic subset of the  affine plane correspond to the so-called simple derivations and there is an extensive literature dedicated to the production of such examples, see \cite{Sha, J, NoT, Moulin, CoS, Bru, No, GLl, SK}, among others.

This is also an active field of study in foliation theory, where one considers the extension of the foliation to the projective plane. A key result in this context is the work of Jouanolou \cite{MR537038} that exhibited a family of examples of foliations without any invariant algebraic curve in the projective plane and deduced the (Baire) genericity of such examples. Compare \cite{loray2003minimal} for a generalization. In the opposite direction, the study of foliations in the neighborhood of invariant divisors provides important information on the foliation \cite{claudon2015compact} and the study of such divisors remains crucial for the study of algebraic or Liouvillian integrability of foliations \cite{cousin2016toward}.

At the turn of the century, the birational geometry of foliations has been developed \cite{BrBG,Mc,Me} and one has a birational classification {\itshape \`a la} Enriques-Kodaira for foliations of projective surfaces.
The goal of this article is to explain how the recent tools in foliation theory allow to classify geometrically the simple derivations  of $\C[x,y]$ and to study their symmetries. We will also present a set of examples found throughout the commutative algebra literature and study their relationships.

In algebra,  a     derivation of the ring $\C[x,y]$ is said to be \emph{simple} if it does not fix globally any nontrivial proper ideal.  It corresponds to   a polynomial  vector field   of  $\mathbb{C}^2$  without zeroes  and without 
algebraic  trajectories.

For any  derivation    $\der$, the isotropy  group   
$\Aut(\der)$  is composed by the  $\C$-automorphisms\\ ${\rho:\C[x,y]\to\C[x,y]}$  which verify  
 \[\rho \der = \der \rho.\]  
Although there exist  derivations  with infinite isotropy group, the  main  result of \cite{partI}    is that     $\Aut(\der)$   \emph{is trivial
for   any  simple derivation}.

Take $\rho\in\Aut(\C[x,y])$, $R: \C^2 \to \C^2$   the polynomial 
automorphism  associated to   
$\rho$  and let 
   $\omega_{\der}$  be   the dual  $1$-form  to the vector field $\der = f\, \partial_x+ g \,\partial_y$ (i.e. $\omega_{\der}  =  g \, dx - f \,dy$). 
   Then   $\rho \der = \der \rho $ is   equivalent  to \[R^*(\omega_{\der} )  = Jac(R) \cdot \omega_{\der}\]
where $Jac(R)\in \C^*$  is the Jacobian determinant of $R$.
A less   restrictive  condition is that
\[ R^*(\omega_{\der} ) = c \cdot  \omega_{\der},
\]
for some  
$c\in\C^*$ (depending  on  $R$). This means
that $R$ preserves  the \emph{foliation} $\calf_{\der}$   of $\C^2$  associated to $\der$ (or to   $\omega_{\der}$), see Remark~$\ref{constant factor}$ p. \pageref{constant factor}. 

We denote $Pol(\calf_{\der})$  the 
group consisting of  polynomial 
automorphisms of $\mathbb{C}^2$  which  
preserve the   foliation $\calf_{\der}$. There  is a natural  homomorphism
\[ \Aut(\der) \hookrightarrow Pol(\calf_{\der}).\]

Let  us denote  $\calf$  the singular holomorphic foliation    of   the  projective plane $\P^2=\C^2\cup L_\infty$
which  is  the  extension of   $\calf_{\der}$  in $\C^2$ .  All along the paper,  if $\der$  is a simple derivation,    both   $\calf_{\der}$ in $\C^2$  and  its extension   $\calf$ in $\P^2$  are called    \emph{foliations 
associated to simple  derivations}.  

But  the reader must be warned that,  even if $\calf_{\der}$ has no singularity, 
some 
singularities  of $\calf$  
along the line at infinity $L_\infty$  are unavoidable, see \cite[Prop. $2.1$]{BrBG}. Also  beware that  the line at infinity    $L_\infty$  may be  
 invariant  by  $\calf$.

Denote  $Bir(\calf)$ the 
group of  birational transformations of  $\P^2$ which preserve a foliation   $\calf$; the elements in $Bir(\calf)$ are sometimes called \emph{birational symmetries} of $\calf$. If  $\calf$  
extends  a foliation $\calf_{\der}$ of $\C^2$,  then there is a natural 
homomorphism 
\[Pol(\calf_{\der}) \hookrightarrow  Bir(\calf)\]
whose meaning is that  a  (non-linear)  polynomial automorphism of $\C^2$   
extends to  a special  type of 
    birational map of $\P^2$.  Namely, a birational map with   a unique (proper) point of indeterminacy $p\in L_\infty$,  whose  net  effect on $\P ^2$ is to 
    replace $L_\infty$  by the 
    strict transform  of  the last exceptional curve 
  introduced   in the elimination  of the  indeterminacy point.

In Section~$\ref{recobrimentos}$, we propose a construction of  simple derivations $\der$ with arbitrary large finite~$Pol(\calf_{\der})$.
This shows the optimality of the following.

\begin{THM}\label{finite}
 Let $\F$  be a foliation of $\Pd$ whose restriction 
 $\F|_{\C^2}$ to $\C^2$  has   no algebraic invariant curve. Then $Bir(\F)$  is finite; in particular, a  foliation associated to  a simple  derivation admits only finitely many birational symmetries.
\end{THM}

Theorem~\ref{finite} is actually derived from the next  result  which determines, in particular, the positions that foliations  associated  to simple derivations may occupy in the   
birational  classification of foliations. This classification is based on the notion of
\emph{Kodaira dimension} of a foliation, denoted $\kappa(\calf)$, whose  range is  $\kappa(\calf) \in \{-\infty, 0,1,2\}$, see  Section~$\ref{Terminology}$.

 For a reduced divisor in a quasiprojective surface, we say it  is a \emph{rational} curve if its projective  closure  has geometric  genus zero. By  
a \emph{Riccati foliation}\label{riccati} on $\Pd$ we mean a foliation  which, up to a birational modification of $\Pd$, is  everywhere transverse  to the general fiber of a rational fibration \label{page_riccati}(see \S \ref{section_riccati} for more details).

\begin{THM}\label{kodsimples}

 Let $\calf$  be a foliation    of the projective plane such  that the restriction 
 $\calf|_{\C^2}$ has no invariant rational curve.

\begin{enumerate}[$($i$)$]

\item\label{kodsimples i}Then  $\kappa (\calf) \geq 1$;
\item \label{kodsimples iii}If  $\calf|_{\C^2}$ has no  invariant algebraic curve, then   $\kappa (\calf)  = 1$ if and only if  $\calf$ 
is a Riccati foliation.

\item \label{kodsimples iv}The  cases  $\kappa (\calf)  \in \{ 1 ,   2\}$ are  realized  by foliations  associated to simple derivations.
\end{enumerate}
\end{THM}
Note that Theorem \ref{kodsimples} applies to a class of foliations which is larger than the one of foliations associated to simple derivations and that case \ref{kodsimples}-($\ref{kodsimples iii}$) includes the foliations associated to \emph{Shamsuddin derivations} \textit{cf.}  \cite{Sha}.

In fact, we obtain Theorem~\ref{kodsimples}-($\ref{kodsimples i}$) as a special case of the following with $(X,D)=(\P^2,L_{\infty})$.
\begin{THM}
Let $X$ be a smooth projective rational surface and let $D$ be  a reduced divisor on $X$.
 Suppose that $X\setminus D$ is simply connected. Then any foliation $\F$ on $X$ that possesses no invariant rational curve outside $D$ satisfies $\kappa(\F)\geq 1$.
\end{THM}

  In  Section \ref{Exemplos} we study the foliations  associated to examples of simple derivations  found throughout the literature and discuss their birational equivalence.

\noindent {\bf Acknowledgements.} We thank   Charles  Favre  and   Jorge  Vit\'orio  Pereira   for useful discussions. 
We are also grateful to the anonymous referee whose criticisms and indications helped us to improve this article.

\section{Preliminaries on foliations}\label{Terminology}
The  paper relies  on  concepts and results of the theory of singularities and birational geometry of foliations on algebraic complex surfaces. We  present some basic facts  in this preliminary section but along the paper,  when necessary, we  refer the reader to the corresponding  sections of 
\cite{BrBG} or  \cite{BrPisa},  where the  theory is    
masterfully explained. 

\subsection{First definitions}On a smooth complex surface $X$, a foliation $\F$ is given by an open
covering $(U_i)$ of $X$ and local vector fields $v_i\in H^0(U_i,TX)$ with \emph{isolated zeroes} such that there exist non vanishing holomorphic functions $(g_{ij})$ on the intersections $U_i\cap U_j$ satisfying 
\begin{equation}\label{cocycle} v_i=g_{ij}v_j.
\end{equation}
The locus defined by the vanishing of the local vector fields $(v_i)$ is called the \emph{singular locus} of $\F$ and denoted $\sing(\F)$. 

The cocycle $(g_{ij})$ defines a line bundle $T^*\calf$ on $X$, its dual is denoted $T\calf$.  Relation~$(\ref{cocycle})$ means that the family $(v_i)$ defines a section of $T^*\calf\otimes TX$ and hence a sheaf map $T\F \to TX$.
Two data $((U_i),(v_i))$, $((U'_j),(v'_j))$ are said to define the same foliation if the images of the associated sheaf maps are the same.

 The line bundle $T\F$ is called the \textit{tangent bundle} of the foliation and its dual $T^*\F$ is the \textit{cotangent bundle} of $\F$.  As defined, the line bundle $T\F$ is not canonically attached
 to $\F$, but its isomorphism class in the \emph{Picard group} $\mathrm{Pic}(X)$ of $X$ is. 

One may also consider foliations on normal singular complex surfaces. They are defined by the datum of a foliation on the complement of the singular locus of the surface.

\subsection{Rational vector fields and $1$-forms}
If $X$ is smooth projective, $T\F$ possesses a non trivial rational section and $\F$ can be given by a rational vector field $\mathcal{X}$, hence in $\mathrm{Pic}(X)$ we have  \[T\F=\mathcal{O}_X(div(\mathcal{X})),\] 
where $div(\mathcal{X})$ denotes the divisor of zeroes and poles of $\mathcal{X}$. 
On a suitable (Zariski) open covering the local vector fields $v_i$ are defined by setting $v_i=h\mathcal{X}_{\vert U_i}$, for a well chosen rational function $h$ on $U_i$. This is how we associate a foliation to a simple derivation: we have a \emph{preferred projective compactification of $\C^2$}, namely $\Pd=\C^2 \cup L_{\infty}, (x,y)\mapsto (x:y:1)$, and a polynomial vector field on $\C^2$ extends to a rational vector field on $\Pd$.

One can also define a foliation by local holomorphic $1$-forms with isolated zeroes $(\omega_i)$ that vanish on the local vector fields $(v_i)$. If $X$ is projective, such a family $(\omega_i)$ is obtained by locally eliminating poles and codimension 1 zeroes of a non trivial rational $1$-form. Hence, on a smooth projective surface, a foliation may be defined by either a  non trivial rational $1$-form or a a non trivial rational vector field.

\subsection{Curves and foliations}
A curve $C$ is termed \emph{invariant by $\F$} or \emph{$\F$-invariant} if it is tangent to the local vector fields defining $\F$.  When a compact  curve $C\subset X$ is \emph{not} $\calf$-invariant we have the very useful \textit{formula}
\[T^*{\calf} \cdot  C = \tang(\calf,C) - C\cdot C,\]
where  $\tang(\calf,C)$ is the   sum of orders of tangency between $\calf$ and $C$, \textit{cf.} \cite[Prop. $2.2$]{BrBG}.  
\subsection{Degree of plane foliations}
If $X=\P^2$ is the projective plane, the \emph{degree} of $\calf$ is $\deg(\calf)\in \Z_{\geq0}$ defined as the number of tangencies of $\F$ with a general projective line. In this case, the previous formula gives 
\[T^*{\calf} = \mathcal{O}_{\Pd}( \deg(\calf) -1 ).\]
Moreover, if $\F$ is induced by a polynomial vector field with isolated zeroes $P(x,y)\partial_x+Q(x,y)\partial_y$ on $\P^2$ then, using the usual total degree for bivariate polynomials,
\[ \deg(\calf) =\left \lbrace
 \begin{array}{l} max(\deg P,\deg Q)\mbox{ if the line at infinity is invariant,}\\
 max(\deg P,\deg Q)-1\mbox{ otherwise.}
 \end{array}
 \right.
 \]

\subsection{Camacho-Sad formula}\label{CSad} 
The self-intersection of a smooth compact invariant curve $C$ can be computed from certain indices of $\F$ along $C$.
For every $p\in C$, consider a local defining $1$-form $\omega$ for $\F$ around $p$. If $C$ has a local equation $f=0$ at $p$, one has a local decomposition $\omega=hdf+f\eta$. Define the Camacho-Sad index $CS(\F,C,p)$ as $Res_p(-\eta/h)$. The Camacho-Sad \textit{formula} is then
\[C\cdot C=\sum_{p\in C}CS(\F,C,p).\]

\subsection{Birational maps and foliations}\label{BirFo}
Let $X$ and $Y$ be projective surfaces with at most normal singularities 
and $\phi:X\dasharrow Y$  a birational map. If we have a foliation $\F$ on $X$ given by the rational vector field $\mathcal{X}$, we can define a foliation $\phi_*\F$ on $Y$ as the one defined by the rational vector field $\phi_*\mathcal{X}$. Conversely, from a foliation $\mathcal{G}$ on $Y$, one defines $\phi^*\mathcal{G}:=(\phi^{-1})_*\mathcal{G}$.  
We say that the foliations $\calf$ and  $\phi_*\F$ are \emph{birationally equivalent} and that $\phi_*\F$ is a \emph{birational model of $\F$}.
\begin{rem}\label{constant factor}
If $\phi :\Pd\dasharrow\Pd$ is induced by a polynomial automorphism $R$ of $\C^2$ and  $\F$ is given by a polynomial vector field $\mathcal{X}$ on $\C^2$ with isolated zeroes, the condition $\phi^*\F=\F$ is tantamount to $R_*\mathcal{X}=h\mathcal{X}$ for a suitable rational function $h$. However, as $R$ is a polynomial automorphism,
the vector $R_*\mathcal{X}$ is a polynomial vector field on $\C^2$, with isolated zeroes.
In particular, the factor $h$ is a constant $c\in\C^*$. A similar reasoning works with polynomial $1$-forms.
\end{rem}

\subsection{Singularities}\label{rednonred}
Around a singular point $p\in X$,  using local centered coordinates $z,w$ the foliation $\calf$ is defined by a holomorphic vector 
field $v = f(z,w) \frac{\partial  }{\partial  z } + g(z,w )\frac{\partial  }{ \partial w}$ with $f(0,0)=g(0,0)=0$ and $\mbox{gcd}(f,g)=1$. The \emph{Milnor number} of the singular point is defined as
\[\mu(p,\calf) = \mbox{dim}_{\mathbb{C}} \frac{ \mathcal{O}_p  }{< f,g >}\] 

Denote $\lambda_1, \lambda_2$ the eigenvalues of the linear part (first jet) of $(z,w)\mapsto (f(z,w),g(z,w))$.  We say that $p$ is a \emph{reduced} singularity of $\F$ if at least one of them, say $\lambda_2$, is not zero and if $\lambda:=\lambda_1/\lambda_2\not \in \Q_{>0}$; otherwise the singularity is \emph{non-reduced}. A special case of non-reduced singularity occurs when the linear part is the identity, in this case $\lambda= 1$, and we say that $p$ is a \emph{radial  point}. 

If $\lambda \neq 0$ we say that the singularity is \emph{non degenerate};  otherwise  we call it a \emph{saddle-node}. The separatrix of a saddle-node which is tangent 
to the non-zero eigenvector is called the \emph{ strong separatrix}. If there is a second separatrix, it is called \emph{weak separatrix}.

We say that $p$ is a \emph{Morse  point} if it is non degenerate and, in suitable coordinates, admits a local  holomorphic  first integral of the form $\phi(z,w) = z^2+w^2+\mbox{h.o.t.}$; note that for Morse  points $\lambda= -1$.

\subsection{Reduced and   relatively minimal models of foliations} A foliation $\calf$ on a smooth surface is said to be {\it reduced} if all its singularities are reduced. After Seidenberg \cite{Sei}, foliations on smooth projective  surfaces always admit a \emph{reduction of singularities}: a birational morphism $\Sigma:M\to X$ obtained as a composition of blowing-ups  such that $\calfp:=\Sigma^*\F$ is a reduced foliation. 

Such a {\em reduced model} $\calfp$ is not unique. Indeed, by performing a blowing-up at either a non-singular point or a reduced singularity the transformed foliation remains reduced. Doing such an ``unnecessary'' blowing-up creates a \emph{foliated exceptional curve} or \emph{$\calfp$-exceptional curve}: a rational curve of self-intersection $-1$ whose contraction to a point $q$ yields a foliated surface with at most a reduced singularity at $q$. A  reduced  model $\calfp$  is called  a  \emph{relatively  minimal model}
when it is free of  \emph{$\calfp$-exceptional curves}.

\subsection{Kodaira dimension} The  \emph{Kodaira  dimension} $\kappa (\calf)$ of $\calf$ is defined by
\[\kappa (\calf )  :=  \limsup_{n\to +\infty}  \, \frac{1}{\log n}\, \log h^0( (T^*{\calfp})^{\otimes n}),\]
for any reduced model $\calfp$.
This is a birational invariant with values in $\{-\infty, 0,1,2\}$.
\subsection{Zariski decomposition}
If  $\calf$ is not birationally equivalent to  a rational fibration,    Miyaoka and Fujita's results   assure  that 
 the cotangent  line bundle   $T^*\calf$ admits a unique so-called {\em Zariski 
decomposition} \cite[Thm. $14.14$ p. $220$]{badescu2013algebraic} \[T^*\calf \equivq  {\bf N}+{\bf P},\] where

\begin{itemize}
\item $\equivq$ means  equality in the group of rational divisors $\mathrm{Div}(X)\otimes_{\Z}\Q$,

\item the  \emph{positive part} ${\bf P}$ is a \emph{nef}  $\Q$-divisor (i.e. ${\bf P}\cdot C\geq 0$  for every curve $C$),

\item the \emph{negative part} ${\bf N}=\sum_j  \alpha_j {\bf N}_j$  is a $\Q^+$-divisor ($\alpha_j\in \Q^+$) and each connected component  of $\cup_j {\bf N}_j$ is 
contractible to a normal singularity,

\item ${\bf P}\cdot {\bf N}_j = 0$, $\forall j$. 
\end{itemize}

The numerical Kodaira dimension of $T^*\calf$ is then defined depending on the numerical properties of $\mathbf{P}$: 
\[\left \lbrace \begin{array}{rl} \nu(T^*\calf)=0 &\mbox{ if } {\bf P}\equiv 0 ,\\
\nu(T^*\calf)=1  &\mbox{ if } {\bf P}\cdot {\bf P}=0 \mbox{ and } {\bf P} \not \equiv 0, \\
\nu(T^*\calf)=2  &\mbox{ if } {\bf P}\cdot {\bf P}>0;
\end{array} \right.\]
where ${\bf P}\equiv 0$ means ${\bf P}$ intersects any divisor trivially.
The numerical Kodaira dimension $\nu(\calf)$ of a foliation $\calf$ is then defined as the numerical Kodaira dimension $\nu(T^*\calfp)$, where $\calfp$ is any reduced birational model of $\calf$.

\subsection{Birational classification}
The birational classification of foliations is done in the spirit of Enriques-Kodaira classification of surfaces, according to the values of $\nu(\F)$ and $\kappa(\F)$. 

We restrict now to mentioning a striking consequence of the classification: $\nu(\F)=\kappa(\F)$ unless $\nu(\F)=1$ and $\kappa(\F)=-\infty$ in which case $\F$ is a   Hilbert modular foliation as they will  be described in Section~\ref{hilbert}.
A few other aspects will be used in the text with suitable references.

\subsection{Nef models of foliations}If   $\calfp$  is a relatively  minimal  (reduced) model  of  $\calf$  and if  $\calfp$  is not a rational fibration,  McQuillan's theorem 
\cite[Thm. $8.1$]{BrBG}  assures  
that the   support 
of ${\bf N} $     in  $T^*\calfp \equivq  {\bf N}+{\bf P} $    is a union of so-called \emph{maximal 
 $\calfp$-chains}.

A  $\calfp$-chain is 
a chain of invariant rational  $(-n)$-curves, with $n\geq  2$, which starts  with a  curve containing just one singularity of $\calfp$ and where the other components,
 if it has more than one,
contain two singularities,  all singularities being reduced and  non-degenerate.  The contraction of a $\calfp$-chain produces a rational surface singularity, more precisely,
a cyclic quotient singularity. The induced foliation on the resulting singular surface is called a \emph{nef model}  of $\calfp$, denoted along the paper by  $\calf_{nef}$.
 
\subsection{Rational fibrations and Riccati foliations}\label{section_riccati}
Consider a rational fibration $p:X\to B$,  $x\in B$ a regular value of $p$ and a foliation $\F$ on $X$. Assume that $\calf$ may be defined by a holomorphic vector field $v$ around the fiber $F:=p^{-1}(x)$.
\subsubsection{}\label{alter fibr/ricc} Let  $U\subset B$ be a coordinate patch  containing $x$ such that $p|_{U}$ is isomorphic to $\pi: F\times U\to U$ and $v$ si holomorphic on $U$. 
For a coordinate $w$ centered at $x$ on $U$ and an affine chart $z$ on $F\simeq \Pu$, one has
\[ (\star)\label{localform}~~~ v=[a(w)+b(w)z+c(w)z^2]\partial_z+d(w)\partial_w,\]
where $a,b,c,d$ are holomorphic on $U$. If there exists an $\F$-invariant section $\Gamma$ of $p$ over $U$, up to an automorphism of  $\pi$ we can suppose it is $z=0$, in which case $a=0$. We see that the vanishing of $v$ on $\Gamma$ implies that $d=0$ and that all the fibers of $p$ are $\F$-invariant.
\subsubsection{} \label{inv fiber} The condition $d\neq 0$ means that $\F$ is everywhere transverse to the general fiber of $p$. In this case, one says that $\F$ is a \emph{Riccati foliation} with respect to $p$ and the fibration $p$ is called 	an  \emph{adapted fibration} for $\F$. The fibers of $p$ which are not everywhere transverse to $\calf$ are called \emph{special fibers}. Untill the end of Section~\ref{monod} we suppose $\calf$ is a Riccati foliation with respect to $p$. We also assume $v$ does not vanish on $F$, which can be arranged up to dividing $v$ suitably.

The $\F$-invariance of the regular fiber $F$ is then characterized by $d(0)=0$.
It is equivalent to the existence of a singularity of $\F$ in the fiber $F$.
\subsubsection{}\label{monod}
Restricting $\F$ to  the complement of the  special fibers $(p^{-1}(x_i))_{i=1,\ldots,n}$, one obtains a flat $\Pu$-bundle and its monodromy representation $\pi_1(B\setminus \{x_1,\ldots,x_n\})\to \mathrm{PSL}_2(\C)$. 
If the basis of the fibration is $B=\Pu$ and there is  only one special fiber $F$, the monodromy is trivial and any  $\calF$-invariant local section at  $F$   extends to a  global invariant  section.

The next paragraph explains the local structure of nef models of Riccati foliations.
\subsubsection{}
Let $\calfp$ be a reduced model of a Riccati foliation $\calf$. After \cite{BrPisa}, up to birational morphism (contraction of curves in $\calfp$-invariant fibers), one reaches a nef model  of the Riccati foliation for which  each fiber has one of the local types $(a),(b),(c),(d),(e)$ described below. The figure is adapted from 
 \cite[p. 20]{BrPisa}.

\[\includegraphics[height= 5.00cm, width= 13.00cm]{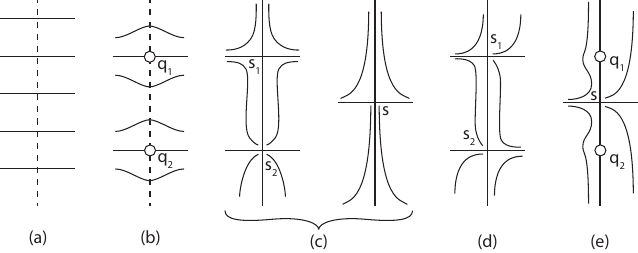}\]

\begin{itemize}
 
 \item[$(a)$] The fiber is regular and transverse to the foliation.
 
 \item[$(b)$] The surface admits two cyclic quotient singularities $q_1,q_2$ of the same order $k\geq 2$ along the fiber. The fiber has order $k$ and is not invariant by the foliation.
 
 \item[$(c)$] There are two possibilities: the fiber is regular and invariant and either contains two non degenerate singularities   $s_1,s_2$  or a unique   saddle-node $s$ with  Milnor number $\mu= 2$  whose strong separatrix is transverse to the fiber.
 
 \item[$(d)$]  The fiber is regular and contains two saddle-nodes
    whose strong separatrices  are contained in the  fiber.
 
 \item[$(e)$]  The surface admits two quotient singularities with order equal to $2$ along the fiber;   there is a saddle  node on the fiber
   whose strong separatrix is
 in the
 fiber. The fiber has order~$2$.
 
\end{itemize}

The cotangent line  bundle of $\F_{nef}$ is obtained as the pull-back of a $\Q$-divisor in the basis, in particular $\kappa(\F)$ is at most $1$. 

If the basis of the adapted fibration  is  $B=\Pu$ and there is  only one special fiber for $\calf_{nef}$, the triviality of the 
monodromy  prevents this fiber from being of type $(b)$ or $(c)$. However, cases $(d)$ and $(e)$ are realized in Examples  \ref{shamgrau8} and \ref{nilp}, respectively.

\subsubsection{}\label{elem}
A key tool in the birational geometry of Riccati foliations are the \emph{elementary transformations}. These are birational modifications of ruled surfaces consisting  of one blowing-up and a successive contraction: after a blow-up at $p$, the total transform of the fiber of $p$ is the sum of  two $(-1)$-curves, contraction of the strict transform of the fiber yields a new ruled surface.
The self-intersection of any section passing through $p$ decreases of $1$ after such a transformation. For other sections the self-intersection increases of one.
 
\section{Proof  of Theorem~C}
In order to prove Theorem~C, we will rule out successively $\kappa (\calf) = -\infty$ and $\kappa (\calf) = 0$.
 
The birational classification of foliations with  $\kappa (\calf) = -\infty$ (\cite{BrPisa} or \cite{Mc})  asserts that this class is composed by 
\emph{rational  fibrations}  and by foliations birationally equivalent to   the  so-called \emph{Hilbert modular foliations}. 
Hence, in order to prove $\kappa (\calf) \neq -\infty$, we only need to exclude Hilbert modular foliations. 
\subsection{Exclusion of Hilbert modular foliations}\label{hilbert}
A  \emph{Hilbert modular  surface}   is  
 defined (following \cite[p. 25]{BrPisa}) as 
a (possibly  singular)  projective surface $Y$  containing a  (possibly empty) 
curve $C\subset Y \setminus \sing(Y)$ such that:
\begin{itemize}
\item each connected component of $C$ is a cycle of smooth rational curves, contractible to a normal singularity; if a connected component of $C$ is irreducible, then  it is a rational nodal curve.

\item $Y\setminus C$  is uniformised  by the bidisc $\Hhyp \times \Hhyp$,
i.e. we have an isomorphism of analytic spaces \[Y\setminus C  \simeq Y_{\Gamma}:=(\Hhyp \times \Hhyp)/\Gamma\]
where $\Gamma$  is a lattice (discrete finite covolume subgroup) of $\PSL(2,\R)^2\subset \Aut( \Hhyp \times \Hhyp  ) $;

\item $\Gamma$  is irreducible  (i.e. does not contain a finite index subgroup of the
form $\Gamma_1 \times \Gamma_2$  with $\Gamma_j \subset \PSL(2,\R)$, $ j = 1, 2$).
 \end{itemize}
The natural singular foliations of the Hilbert modular surface $Y$  which come from the horizontal and vertical foliations by discs of  $\Hhyp \times \Hhyp$ are called \emph{Hilbert modular foliations}. Both foliations
 leave invariant the curve $C$
 and, in the desingularization of the surface, they leave invariant the exceptional divisors.

\begin{pro}
 Let $X$ be a smooth projective surface and let $D$ be  a reduced divisor on $X$ such that  $X\setminus D$ is simply connected. 
Let  $\calf$   be a   foliation of $X$   such that   $\calf|_{X\setminus D}$  has no  invariant  rational curve. 
Then   $\calf$  is not 
 birationally equivalent to a   Hilbert modular  foliation.  
 
\end{pro}

\begin{proof}
Suppose by contradiction that $\F$ is birationally equivalent to a  Hilbert modular foliation.
Let  $\overline{\calf}$  be a  relatively minimal model of $\calf$ on a suitable smooth projective surface $M$. Then $M$ is the minimal  
desingularization  of a  Hilbert modular surface $Y$ (\textit{cf.} \cite[Thm. $5.1$]{BrBG})  and $\overline{\calf}$  is the transform of  a Hilbert modular foliation on $Y$.

As described above, one has $Y=Y_{\Gamma}\sqcup C$, for an irreducible lattice $\Gamma$ of $\PSL(2,\R)^2$ and $C$ a contractible cycle of rational curves in $Y$.

Let \[U:=X \setminus (D\cup \sing(\calf)).\]
Then, as it contains no singularity of $\F$ nor invariant rational curves, $U$  is isomorphic to a  non-empty Zariski open subset of $M$  and  is also isomorphic  to    a non-empty Zariski open subset of 
$\dot{Y}_{\Gamma}$, the complement
of the quotient singularities of $Y_{\Gamma}$. The group homomorphism \[\pi_1(U)\rightarrow \pi_1(\dot{Y}_{\Gamma})\]
induced by the injection  is surjective.
As $U$ is simply connected, so is $\dot{Y}_{\Gamma}$.

By the irreducibility of $\Gamma$, the subset $F\subset\Hhyp \times \Hhyp$ given by the points that have 
non trivial stabilizer under $\Gamma$ is discrete (\textit{cf.} \cite{Shi}). The action of $\Gamma$ on $\Hhyp \times \Hhyp$ can be restricted to an action on 
$(\Hhyp \times \Hhyp)\setminus F$. 
The quotient map \[(\Hhyp \times \Hhyp)\setminus F\rightarrow [(\Hhyp \times \Hhyp)\setminus F]/\Gamma=\dot{Y}_{\Gamma}\] is
then a covering, in the strict sense of topologists. As $(\Hhyp \times \Hhyp)\setminus F$ is simply connected, 
this proves that the fundamental group of $\dot{Y}_{\Gamma}$ is isomorphic to $\Gamma$. 
But $\dot{Y}_{\Gamma}$ is simply connected, so  $\Gamma$ is trivial, a contradiction.
\end{proof}

This result is  sharp if $(X,D)=(\P^2,L_\infty)$: a pair of  Hilbert modular foliations  with exactly  one  invariant rational quintic curve in $\C^2$ is given in \cite{MePe}. It gives rise to an irreducible nodal curve invariant by the reduced models of the  foliations.  
\subsection{Exclusion of vanishing Kodaira dimension}
We procede by contradiction and suppose we have a foliation $\calf$ that satisfies the hypotheses of  Theorem~C and~$\kappa (\calf ) = 0.$
Let  $\Sigma: M \to X$ be a reduction of  singularities 
of $\calf$ and denote the resulting reduced  foliation on $M$  by $\calfp$.

Denote by $q: M \to M^{\prime}$
a finite (possibly trivial) sequence of blowing-downs  of     foliated  exceptional curves, such that  the foliation 
$\calfp^{\prime} := q_*\calfp$
 is  a relatively  minimal model of $\calf$.
By a theorem of McQuillan \cite[Thm. $9.2$]{BrBG}, we have $\nu(\F)=0$.
We can then apply the following, which is actually the simplest case of another  theorem of McQuillan \cite[Thm. $8.2$]{BrBG}.

\begin{lem}
Let $X$ be a simply connected projective variety endowed with  a relatively minimal foliation $\G$ such that $\nu(\G)=0$. Then $\G$ can be  defined by a holomorphic vector field with isolated zeroes.
\end{lem}
\begin{proof}
Let $T^*\G\equivq\mathbf{N}+\mathbf{P}$ be the Zariski decomposition of $T^*\G$.
Since $X$ is simply connected, we have $h^1(X,\C)=0$, in particular, $h^1(X,\mathcal{O}_X)=0$, by Hodge theory. Under these conditions, numerical equivalence of divisors on $X$ is the same as equality modulo torsion. The numerical triviality of ${\bf P}$ then implies  $T^*\G \equivq {\bf N}$.  The non trivial coefficients of ${\bf N}$ are not integers \cite[Addendum p. $100$]{BrBG} hence, $T^*\G$ being an integral divisor, ${\bf N}=0$ and $T^*\G$ is a torsion line bundle. In particular, if $n$ is the order of torsion, from the cohomology exact sequence induced by \[0\to\Z/n\Z\to\calo_X^*\stackrel{(\cdot)^n}{\longrightarrow}\calo_X^*\to 0\] on sees $T^*\G$ is in the image of the map $H^1(X,\Z/n\Z)\rightarrow H^1(X,\calo_X^*)$. The simple connectedness of $X$ yields~$T^*\G=0$ and  trivial tangent bundle $T\G=0$. Therefore $\G$ can be  induced by a holomorphic vector field with isolated zeroes.
\end{proof}

The foliation $\calfp^{\prime}$ of $M^{\prime}$ is hence induced by a holomorphic vector field $\mathcal{X}$ with isolated zeroes.
Let $U\subset M^{\prime}$ be the isomorphic image of $X\setminus(D\cup\sing(\F))$ under $q\circ\Sigma^{-1}$.
Following the argument of Brunella in \cite[pp. $77-78$]{BrBG}, one sees that with an additional birational map $M^{\prime} \dasharrow \Pu\times \Pu$ which induces an isomorphism between $U$ and its image $V$, one transforms the pair $(M^{\prime},\mathcal{X})$ in $(\Pu\times \Pu, \mathcal{X}_1\oplus \mathcal{X}_2)$, where  the vector fields $\mathcal{X}_j$ are holomorphic vector fields on $\Pu$.

More precisely, because of its rigidity, every $(-1)$-curve in $M'$ must be invariant by the holomorphic vector field $\mathcal{X}$. Consequently, contracting such a curve we obtain a surface on which $\mathcal{X}$ induces a still holomorphic vector field. Performing a maximal sequence of such contractions, one reaches a pair $(S,\mathcal{X}_S)$ where $\mathcal{X}_S$ is a holomorphic vector field on the surface $S$. By construction $S$ is a minimal rational surface, it is hence $\Pd$ or a Hirzebruch surface $\mathbb{F}_n$. The subset $U\subset M^{\prime}$ does not contain invariant rational curves so that it is mapped isomorphically on its image $V_S$ in~$S$. If $S=\Pd$, the blowing-up of  any singularity of the foliation then transforms  $\Pd$ in the first Hirzebruch surface, where the pull-back of 
$\mathcal{X}_S$ is still holomorphic. The blowing-up is an isomorphism over $V_S$, since $V_S$ contains no singularity.
In this way, we obtain a birational map $\varphi:M^{\prime} \dasharrow \mathbb{F}_n$ which maps $U$ isomorphically to its image and such that $\calR=\varphi_*\F$ is generated by a global holomorphic vector field $\mathcal{Y}$  on $\mathbb{F}_n$.

If $n=0$, we have the announced map. Suppose $n>0$. The negative section $\Gamma$ is invariant by the vector field (again, due to its rigidity). We cannot have $\mathcal{Y}_{\vert \Gamma}=0$, for in this case  the fibers of the rational fibration $p$ on $\mathbb{F}_n$ would all be $\calR$-invariant (see \ref{alter fibr/ricc}) and  $\F$ would be a  pencil of rational curves. Hence, $\calR$ is a Riccati foliation with respect to $p$. One must have an invariant fiber $F$ for $\calR$: otherwise, the absence of monodromy forces $\calR$ to be a pencil of rational curves. Hence $U$ is mapped  by $\varphi$ to a simply connected Zariski open subset $V_n$ in \[C_n:=\mathbb{F}_n\setminus(\Gamma\cup F).\]
Under a suitable isomorphism $\C^2\simeq C_n$, the restriction  of the fibration $\mathbb{F}_n\to \Pu$ to $C_n$ corresponds to the first projection of $\C^2$.
Any fiber $\tilde{F}$ of $\mathbb{F}_n\to \Pu$ which is distinct from $F$ must intersect $V_n$; otherwise we would have a surjective map $0=\pi_1(V_n)\to\pi_1(C_n\setminus \tilde{F})=\mathbb{Z}$.  Hence  $F$ is the unique $\calR$-invariant fiber for this projection, since $\calR$ has no invariant rational curve in $V_n$.

By Remark~$\ref{second sing}$ (Appendix), we have a sequence of elementary transformations (see \ref{elem}) centered at singularities of the foliation that transforms $\mathbb{F}_n$ in $\Pu\times\Pu$ and preserves the holomorphicity of the vector field.
As the centers of the elementary transformations are singularities of the foliation,  $V_n$ is mapped isomorphically to its image $V\subset \Pu\times\Pu$.  The resulting vector field on $\Pu\times \Pu$ is holomorphic and hence has the form $\mathcal{X}_1\oplus \mathcal{X}_2$, where both terms are holomorphic vector fields on $\Pu$.

 The zeroes of $\mathcal{X}_j$ ($j=1,2$) give $\calR$-invariant fibers of the $j$-th projection. In particular, as the simply connected Zariski open subset $V$ cannot embed in the complement of two distinct such fibers, the vector fields $\mathcal{X}_j$ both have a unique zero. In  suitable affine charts $z_j$, they take the form
 $\mathcal{X}_j=\partial_{z_j}$. This is a contradiction, since the vector field $\partial_{z_1}+\partial_{z_2}$  leaves invariant  the rational curves given by the levels of $z_1-z_2$. 

\section{Proof  of Theorem~$B$-$(ii)$   }
Let $\F$ be a foliation that satisfies the hypotheses of Theorem~\ref{kodsimples}-($\ref{kodsimples iii}$)  and has Kodaira dimension $\kappa (\calf) = 1$.
We prove here that $\F$  must be a Ricatti foliation.
The other implication of Theorem \ref{kodsimples}-($\ref{kodsimples iii}$) follows from Theorem\ref{kodsimples}-($\ref{kodsimples i}$) and the fact that Riccati foliations have Kodaira dimension at most $1$.

The foliation $\calf_{|\C^2}$ has no invariant algebraic curve, therefore is not birationally conjugate to a fibration. According to the birational classification  of foliations \cite[Thm. $9.1$]{BrBG}, $\calf$ is either a Riccati foliation or a \emph{turbulent 
foliation}; recall that the definition of turbulent foliation is obtained from the one of Riccati foliation by replacing ``rational fibration'' with ``elliptic fibration'' (see page \pageref{riccati}).

Therefore, proving  Theorem~$\ref{kodsimples}-(\ref{kodsimples iii})$ amounts to excluding the turbulent case.

This will be done in Proposition~$\ref{exclusionT}$, using the notion of a \emph{transversely affine} foliation. 
Consider a foliation $\calf$ on a surface $X$  given by  $\omega=0$ where $\omega$ is a rational $1$-form on $X$. 
We say that $\calF$ is \emph{transversely affine} if there exists a \emph{closed} rational $1$-form $\eta$ such that \[d\omega=\omega \wedge \eta.\]

\begin{rem}
If $\tilde{\omega}= g\, \omega$ is another $1$-form defining $\calf$, then $\tilde{\eta}:=\eta-dg/g$  is closed and satisfies $d\tilde{\omega}=\tilde{\omega }\wedge \tilde{\eta}$,
so that this definition is independent of the defining $1$-form $\omega$.
\end{rem}
The following has already been noticed in \cite[Prop. $22$]{LibroJorge}, we give a slightly different proof.
\begin{pro}\label{propturb}
Every turbulent foliation is a transversely affine foliation.
\end{pro}

\begin{proof}
As transversely affine structures may be transported by birational transformations, stable reduction \cite[Prop. $4.6$]{BrBG} and \cite[Thm. $2.21$ p. $37$]{CLNL}
reduce the proof to the case where the foliation $\calf$ is transverse to the general fiber of  an elliptic fiber bundle $\pi : X \rightarrow C$.

Let $X^*$ be the complement of the set of invariant fibers, and $F=\pi^{-1}(b)\subset X^*$ a fiber. Using the foliation to  identify 
nearby fibers, we obtain a multiform submersion  $\widetilde{X^*}\rightarrow F$ which defines the foliation; it lifts to a submersion
$f :\widetilde{X^*}\rightarrow \C$ to the universal cover of $F\simeq \C/\Lambda$. By construction the monodromy group of $f$ fixes the 
lattice $\Lambda$ and must lie in the group $\mathrm{Aff}(\C)$ of affine morphisms. Hence the monodromy of $df$ is linear (contained in $\C^*$). In particular, if $v$ is a rational
vector field on $X$ which is not tangent to $\calf$, the meromorphic function \[g= df(v):\widetilde{X^*}\rightarrow \C\] has the same monodromy as $df$ and  
\[\omega:=\frac{df}{g}\]
is a well defined meromorphic $1$-form on $X^*$, tangent to $\calf$. We have \[d\omega=-\frac{df \wedge dg}{g^2}  =\omega\wedge\eta,\]
with $\eta=  - \frac{dg}{g}$ a well defined closed meromorphic one form on $X^*$.

 It remains to show that the pair $(\omega,\eta)$ extends meromorphically in the neighborhood of any $\calf$-invariant fiber of $\pi$. 
 Let $U\simeq \D\times F$ be such a neighborhood, $\D$ a disc. Let $(z,w) \in \mathbb D\times \C$ represent the elements of $U$, $z=0$ corresponding 
 to the invariant fiber. We have a local equation of the form \[dw = \frac{dz}{A(z)},\] for $\calf$, with $A(z)$ holomorphic in $\D$.
Let $b$ be a point in $\D$. If the coordinate $w$ is well chosen, in $\D^*\times F$, the submersion $f$ expresses as 
\[f(z,w)=w-\int_b^z\frac{ds}{A(s)}\]
and \[df =  dw-\frac{dz}{A(z)}\]
is meromorphic at $z=0$, and so is $g$; we have the required extension property.
\end{proof}

\begin{pro}\label{exclusionT}
Let $\calf$ be a foliation on $\P^2\simeq \C^2 \cup L_\infty$.
If $\calf$  is a turbulent foliation with $\kappa(\calf) =1$, then it possesses an invariant algebraic curve outside $L_\infty$.

\end{pro}
\begin{proof}
By contradiction, suppose $\calf$ possesses no invariant algebraic curve in $\C^2=\P^2\setminus L_{\infty}$.

By Proposition \ref{propturb}, $\calf$ is transversely affine.
The assumption on non existence of invariant algebraic curves allows to use \cite[Corollary B]{CP} to infer that $\calf$ is given  by the pullback $\omega$ of a $1$-form 
$$\omega_0 = dy + (a(x) + b(x)\, y) \, dx, \quad a,b\in\C[x]$$ 
under a polynomial map $\C^2\rightarrow \C^2$, which 
extends as a rational map $H_0: \Pd\dasharrow \Pd$.

Denote $\calg$ the foliation of $\P^2$ induced by $\omega_0$.

 There exist sequences of blowing-ups $\Sigma_X:X\rightarrow \P^2$, $\Sigma_Y:Y\rightarrow \P^2$ in the source and the target of $H_0$ such that the following conditions are met.
\begin{itemize}
\item The foliations $\overline{\calf}:=\Sigma_X^{*}\calf$ and $\calR:=\Sigma_Y^{*}\calg$ have at most reduced singularities.
\item There exists an elliptic fibration $f_X: X\rightarrow \P^1$ adapted to the turbulent foliation $\overline{\calf}$. 
\item There exists a rational fibration $f_Y: Y\rightarrow \P^1$ adapted to the Riccati foliation $\calR$. 
\item The rational map $H:X\dasharrow Y$ such that $\Sigma_Y\circ H=H_0\circ \Sigma_X$ is actually a morphism (\textit{i.e.} holomorphic).
\end{itemize}

As $\calf|_{\C^2}$, $\calg|_{\C^2}$ possesses no algebraic invariant curves. 
Hence, by  the already proved item   ($\ref{kodsimples i}$)   of Theorem \ref{kodsimples},   we must have $\kappa(\calR)=\kappa(\calg)\geq 1$.
As for every Riccati foliation, we have  $\kappa(\calR)\leq 1$.  Consequently
\[\kappa(\calR)=1.\]

By Lemma $\ref{lemIitaka1}$ below, $f_{Y}$ is the Iitaka fibration (\cite[p. $107$]{BrBG})   of the cotangent  divisor  $T^*\calR$.  
Similarly $f_{X}$ is the Iitaka fibration of the  cotangent  divisor $T^*{\overline{\calf}}$.

From the remark in \cite[p. $20$]{BrBG} it follows that $T^*{\overline{\calf}}=H^*(T^*\calR)\otimes \calO_X(D)$ for  an effective divisor $D$  on $X$ 
(see also the proof of \cite[Lemme $3.2.8$]{Tou}) 

Then, Lemma \ref{lemIitaka2} below yields that  $H$ maps the fibers of $f_X$ in the fibers of $f_Y$: for general $c\in \P^1$, there exists
$r(c)\in \P^1$ such that $H\left( f_X^{-1}(c)\right)\subset f_{Y}^{-1}(r(c))$.

Consider, for general $c$, the following restriction of $H$, \[H_c :f_X^{-1}(c)\rightarrow f_{Y}^{-1}(r(c)).\] 
Denote $R\subset X$ and $B\subset Y$ the ramification and branching curves of $H$, namely \[R:=\{x\in X;  \rk(d_xH)\neq2\},\quad  B:= H(R).\]
The map $H_c$ is  \'etale outside $R$. 

Let $\overline{L}_{\infty}$ be the strict transform of ${L}_{\infty}$ in the sequence of blowing-ups $\Sigma_Y$.
 If $B$ has $\calR$-invariant components, they must be contained in $\overline{L}_{\infty}$ or in the exceptional divisor of $\Sigma_Y$,
 because $\calg$ possesses no invariant curve in $\C^2$. Denote $B_{inv}$ the union of these components. 
 Notice that the general fiber of $f_Y$ intersects $B_{inv}$ at most once, because $f_Y$ is induced by the coordinate  fibration $x$ on $\C^2$.

We assert  (\textbf{A}): \emph{for a general $c$, the curve $f_{Y}^{-1}(r(c))$ does  not  intersect $B\setminus B_{inv}$}. 
 
From (\textbf{A}), we obtain that, for general $c$, the map $H_c$ ramifies at most over  
  one point of $f_{Y}^{-1}(r(c))\simeq\P^1$, contradicting that $f_X^{-1}(c)$ is elliptic.

We conclude by proving Assertion (\textbf{A}).  First,  remark that any non $\calR$-invariant component of $B$  is a curve transverse to the general fiber of $f_Y$ and has a 
 finite number of tangencies with $\calR$.  Therefore   we have, for general $c$:
 
 \begin{enumerate}
\item \label{Cond1} For any point $p' \in \big(f_{Y}^{-1}(r(c))\setminus B_{inv}\big)\cap B$, $B$ is transverse to both $f_{Y}^{-1}(r(c))$ and $\calR$ at $p'$;
\item  \label{Cond2} for any point $p\in  f_X^{-1}(c) \cap R$, $H$ writes as  $(s,t)\mapsto (S,T)=(s^{\ell},t)$,  with  $\ell>1$,  in suitable 
local coordinates $(s,t)$ and $(S,T)$ centered at $p$ and $p'$ respectively.
\end{enumerate} 
 Take $c$ such that we have $(\ref{Cond1})$ and $(\ref{Cond2})$. Suppose we have a point $p'=H(p)$ in $f_{Y}^{-1}(r(c))\cap B\setminus B_{inv}$.
In the adapted coordinates $(S,T)$ of $(\ref{Cond2})$, $S=0$ is a local equation for $B$. By $(\ref{Cond1})$, $\calR$ is transverse to $B$ at $p'$.
Hence we have a local graph $T=\lambda_1 S +o(S), \lambda_1\in \C$, which is tangent to $\calR$. Similarly, the fiber $f_{Y}^{-1}(r(c))$ passing through $p'$ expresses locally as $T=\lambda_2 S +o(S), \lambda_2\in \C$.
In the neighborhood of $p$, the pulled-back graphs have equations $t=\lambda_i s^{\ell}+o(s^{\ell})$ and are tangent at $(s,t)=(0,0)$ because $\ell>1$. One is tangent to $\calfp$ and the other is tangent to $f_X^{-1}(c)$. This implies that $p$ is a tangency point between $\calfp$ and $f_X^{-1}(c)$.
As $f_X$ is  an adapted fibration for $\calfp$, such an intersection point $p'$ cannot exist for $c$ general enough.
 \end{proof}
For the reader's convenience, we prove two facts that belong to  the  birational theory of foliations and varieties. 

\begin{lem}\label{lemIitaka1}

Let $\calfp$ be a reduced foliation on a projective manifold $X$, with $\kappa(\calfp)=1$.
Suppose $\calfp$ is a Riccati or a turbulent foliation, with adapted fibration $f : X\rightarrow C$.
Then $f$ is the Iitaka fibration of $T^*\calfp$.
\begin{proof}
Let $F$ be the general fiber for a fibration $f$ adapted to $\calfp$. Lemma \ref{lemIitaka2} shows that the Iitaka 
fibration associated to $F$ is the fibration $f$. The proof of \cite[Thm. $9.1$]{BrBG} shows $T^*\calfp^{\otimes m}=\mathcal O(nF+D)$ for an 
effective divisor $D$ and suitable integers $m,n>0$.
Lemma \ref{lemIitaka2} (with $\mathcal{L}=id_X$)  allows to deduce that both divisors $F$ and $T^*\calfp$ have the same Iitaka fibration, yielding the conclusion.
\end{proof}
\end{lem} 

In our context, the next  lemma should  be applied  in the case of (foliated)  Kodaira dimension~$1$.

\begin{lem}\label{lemIitaka2}
Let $\mathcal{L} :X_1\rightarrow X_2$ be  a morphism between projective manifolds.
Let $D_1,D_2$ be divisors on $X_1$ and $X_2$, respectively. Suppose these divisors have equal positive  Iitaka dimension. 
Take $k>0$ big enough so that $p_i : X_i\dasharrow \mathbb{P}\Gamma(X_i,\mathcal O(D_i)^{\otimes k})^*$ is the Iitaka fibration of $D_i$, $i=1,2$.
Suppose $D_1=\mathcal{L}^*D_2+D$ with $D$ effective. Let $ r \circ q$ be the Stein factorization of  $p_2 \circ \mathcal{L}$. Then $q$ is the Iitaka fibration of $D_1$.
\end{lem}
\begin{proof}
Choosing a nontrivial global section  $s\in \Gamma(X_1,\mathcal O(D))$ we have an injection 
\[\phi_k : \Gamma(X_2,\mathcal O(D_2)^{\otimes k}) \rightarrow \Gamma(X_1,\mathcal O(D_1)^{\otimes k})\]
\[\sigma \mapsto (\mathcal{L}^*\sigma) \otimes s^{\otimes k}\]
and the following  diagram commutes, with $\phi_k^*$ onto.

\[\xymatrix{
   X_1 \ar[d]_{\mathcal{L}} \ar@{-->}[r]^-{p_1}& \mathbb{P}\Gamma(X_1,\mathcal O(D_1)^{\otimes k})^* \ar[d]^{\phi_k^*} \\
   X_2 \ar@{-->}[r]^-{p_2} & \mathbb{P}\Gamma(X_2,\mathcal O(D_2)^{\otimes k})^*
 }\]
  
  Restricting the maps, with $S=p_1(X_1)$, $T=p_2(X_2)$, we get the following.
 \[\xymatrix{
    X_1 \ar[d]_{\mathcal{L}} \ar@{-->}[r]^-{p_1}& S \ar@{->>}[d]\\
    X_2 \ar@{-->}[r]^-{p_2}& T
  }\]
The map $S\rightarrow T$ is onto. As $\dim S=\dim T$, it must be a generically finite map. The uniqueness of the  Stein factorization
yields $q=g \circ p_1$ for some birational map $g : S\rightarrow S'$. This yields the conclusion, 
because the Iitaka fibration is defined only up to  birational transformations  in the target.
\end{proof}

To explain the limits of Proposition~$\ref{exclusionT}$, we present a turbulent  foliation with $\kappa (\calf)=1$ having 
exactly one rational invariant  curve in $\C^2$.  
\begin{exa} 
Consider the  degree $4$  foliation $\calf$ on $\P^2$ associated to the 1-form on $\C^2$ given by
\[\omega = d(y^2+x^3)   + (y^2+x^3)\cdot (3 y dx - 2 x dy).\]
The foliation  $\calf$ leaves invariant the cuspidal rational  cubic $C:\, y^2 + x^3  = 0$
and the line at infinity   $L = L_{\infty}$.  From \cite[Lemme IV.$2$]{Lo} 
it follows that these are the   unique  algebraic $\calf$-invariant curves. 
The  pencil of cubics $\mathcal{E}$  
generated  by   $C$
and the line  at infinity  $L$
(taken with  multiplicity  $3$) shall  give  rise to the adapted elliptic fibration for the turbulent foliation $\overline{\calf}$,  as obtained from $\calf$ after reduction of singularities.

To see this,   remark that
\[\omega \wedge d(y^2+x^3) = 6 (y^2+x^3)^2\, dx \wedge dy\]
which  means that the contact locus  of $\calf$ and $\mathcal{E}$ is exactly    $C \cup L$. A reduced model $\overline{\calf}$   and the adapted elliptic fibration  are obtained  after
$9$  blowing-ups at $L\cap C$ (and  infinitely near points), besides  $3$ extra
 blowing-ups at the cuspidal point of $C$
(and  infinitely near points).  The elliptic  fibration  has just two
 singular fibers: one of type $II^*$, in Kodaira's notation,  and  one which is a blown up fiber  
of type  $II$, as shown in  the next figure, where $E_9$ is a section. Except $E_9$ and $E_3$, all the components of the exceptional divisor are $\overline{\calf}$-invariant.
\[\includegraphics{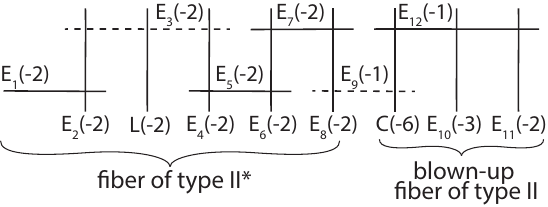}\] 
The possibilities for  turbulent  foliation are $\kappa (\calf) = 0, 1,-\infty$. The third option is trivially excluded, because it corresponds  to foliations given by rational fibrations.
We will show $\kappa (\calf)~\neq~0$ by contradiction.

Since  $\overline{\calf}$ is free of  foliated exceptional curves,  if  $\kappa(\calf )= 0$ then 
$T^{*} \overline{\calf} \equivq {\bf N}$. The
 support of the negative part  ${\bf N}$ is contained in the six 
  $ \overline{\calf}$-chains, which are given by:
i) $\overline{L}$, ii) $\overline{E}_1$, $\overline{E}_2$, iii) $\overline{E}_4,\overline{E}_5,\overline{E}_6,\overline{E}_7, \overline{E}_8$, 
iv) $\overline{C}$, v) $\overline{E}_{10}$ and  vi)
$\overline{E}_{11}$.

Now we assert  that there exists a rational fibration  whose generic fiber $F$ has  tangency with $\overline{\calf}$ along the curve $\overline{E}_3$. Based on this,  we obtain
\[ T^{*} \overline{\calf} \cdot F  = tang(\overline{\calf},F) - F \cdot F = tang(\overline{\calf},F)  > 0,\]
although $F$ does not intersect any $ \overline{\calf}$-chain, a  contradiction.

To justify the assertion,  consider the pencil of rational 
cuspidal  cubics  in the projective  plane  $\mathcal{C}_{\lambda}:\,   y^2   = \lambda x^3$.
There are  two base points in the plane at   $(1:0:0)$ and   $(0:1:0)$.  
Three   blowing-ups at  $(1:0:0)$  (and infinitely near points) and three blowing ups at $(0:1:0)$  (and infinitely near points)  are enough  to separate the cuspidal cubics  of $(\mathcal{C}_{\lambda})$  and produce a rational fibration. 
  The third  blowing-up at infinity  introduces the exceptional line $E_3$ (figure above). An explicit computation shows that, in local 
coordinates where $E_3:  (x=0)$ and  
the  rational fibration is associated to
the $1$-form $dy$, the transformed foliation of $\calf$  is associated to
\[ (6 x^4 y^4 + 6 x^4 y^3)\, dx + (2 x^5 y^3 + 3 x^5 y^2 + y + 1)\, dy   \]
So there is tangency along  $E_3$, as asserted.
\end{exa}

\section{Birational geometry of examples and proof of  Theorem~$B$-$(iii)$}\label{Exemplos}

In this section we give examples of foliations $\calf$ of $\P^2$ which are associated to simple derivations and whose Kodaira dimension satisfies $\kappa (\calf ) \in \{1,2\}$. 

In every case we describe the corresponding reduction of singularities and give a nef model. Moreover, we show some birational (non-)equivalences between examples.

We  provide diagrams  to illustrate the  reduction of singularities and nef models. The  following conventions are used  in  the diagrams.
\begin{itemize}
  \item The affine coordinates $(x,y)\in \C^2$ correspond to  $(x:y:1) = (x_0:x_1:x_2)\in \Pd$.
 \item We denote as $(\sigma_i)$ the blowing-ups of points composing a reduction
  of singularities of a given foliation $\calf$. The blowing-downs composing the morphism $M\to S$ to a nef model $\calf_{nef}$ will be denoted as $(q_j)$, when contracting $(-1)$-curves, and as   $(\rho_i)$ otherwise.

 \item In the figures, we use  $nd$, $sn$, $m$, $r$, $nil$  for \emph{non-degenerate}, \emph{saddle-node}, \emph{Morse}, \emph{radial}   and \emph{nilpotent}  singularities, 
 respectively (\textit{cf.} Section \ref{Terminology}).
 \item The line at infinity is   denoted $L$ and its  strict transforms denoted as $\overline{L}$, except in the figures (see the next point).
 \item In the figures, we use the same symbol for a curve and its strict  transforms under extra blowing-ups, but the self-intersection 
 numbers  indicated in parentheses $(n)$ will help to avoid confusions.
 \item  The bracket $[m]$  denotes  the  polar order of the fiber ($z=0$)  of the Riccati  foliation locally defined by $z^m dw + ( a(z) w^2 + b(z) w + c(z)      ) \, dz $, with $a,b,c$ holomorphic at $z=0$.
\end{itemize}

We   start   with      examples having   $\kappa (\calf) =1$.
Part of these examples are associated to \emph{Shamsuddin derivations}, see \cite{Sha}. These are derivations of the form \[\der=\partial_x+(a(x)y+b(x))\partial_y,~~~~a,b\in\C[x].\]
The associated foliation is given by  $\omega=dy-\left(a(x)y+b(x)\right)dx$ and is a special Riccati foliation. 
\begin{rem} For a Riccati foliation $\F$, the Kodaira dimension always satisfies $\kappa (\calf) \leq 1$. In view of our Theorem \ref{kodsimples}, if such a foliation is associated to a simple derivation of $\C[x,y]$,
it must satisfy $\kappa (\calf)=1$. However, for illustrative purposes we present direct computations of the Kodaira dimension for such examples. 
The fibrations adapted to the examples we propose correspond to pencils of parallel lines in $\C^2$. Notice however that, applying polynomial automorphisms of $\C^2$, we can transform these lines in curves of arbitrary degree, the transformed derivation remaining simple.
\end{rem}
\begin{exa}\label{shamquadratico}

Consider the foliation $\calf$  of the projective  plane   associated to 
 \[ \omega = (x y + 1 ) \,  dx - dy ,\]
 called   Bergman's  example  in \cite{CoS}.
The extended  foliation   $\calf$ of $\P^2$   has   degree $2$.
The  point at infinity  $(1:0:0)$ 
 is a saddle-node with Milnor number $\mu=3$
 whose strong separatrix is the line at infinity $L: x_2=0$; in particular $L$ is $\calf$-invariant. At $(0:1:0)$  there is a a quadratic singularity:
 the  blowing-up  at this point
 produces a  Riccati foliation on $\mathbb{F}_1$,  leaving invariant the exceptional curve $E_1$. There is just one singular point along $E_1$,
 a saddle-node with Milnor number $\mu=3$, strong separatrix $\overline{L}$  and weak separatrix $E_1$. 
 This already gives a nef model in for $\calf$ in $\mathbb{F}_1$.  Taking into account the multiplicity of $\overline{L}$ (thought of) as an invariant fiber, and the formula in \cite[p. 48]{BrBG} for the cotangent  bundle we deduce
 \[  T^*{\calf_{nef}}  =  \calo(-2 \overline{L}) \otimes \calo(3\overline{L}) =  \calo(\overline{L})\quad\mbox{and}\quad
 \kappa(\calf) = 1.    \]
 \[\includegraphics{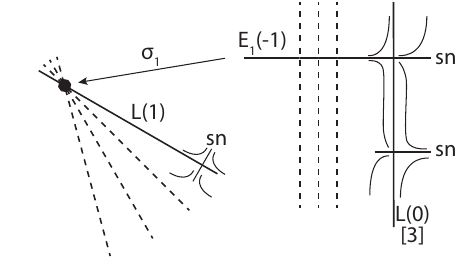}\] 
 \end{exa}

\begin{exa}\label{GL}

From \cite{GLl} we consider the foliation of degree $2$ in the projective plane associated to 
\[ \omega =  (  1 + x \, (2 x + y)    )  dy + 2 x ( 2 x + y   )  dx = 0    \]
The singularities  along $L$ are a saddle-node at  $(0:1:0)$
and a  non-reduced (quadratic) singularity  at $( -1 :  2  : 0)$.
Any affine line $2 x + y = c $, $c\in \C$, is completely  transverse  to the foliation.  One blowing-up at 
$( -1 :  2  : 0)$  is enough to  reduce the singularity and produces a Riccati foliation.
The exceptional line is invariant  and  has a  saddle-node. 
$\overline{L}$  is  the unique  invariant  fiber.  
This is already a nef model.   The multiplicity  of $\overline{L}$ as an invariant fiber  is  $3$ and we obtain  
\[T^*\calf_{nef}  =  \calo(-2 \overline{L}) \otimes \calo(3\overline{L}) =  \calo(\overline{L})\quad  \mbox{and}\quad  \kappa(\calf) = 1\]

\end{exa}

\begin{exa}\label{shamgrau3}
 
We consider the foliation  of degree $3$ in the projective plane  of Shamsuddin  type given by (see \cite[Exa. 13.3.5]{NoT})
\[ \omega =  ( y  x^2 +  x y + x^2 )  dx - dy = 0.\]
At $(1:0:0)$  
there is a saddle-node with Milnor number $\mu=4$, whose strong separatrix is $L$.
At  the vertical infinity there is  a cubic singularity.  The  blowing-up at this point produces
a Riccati foliation  relatively  to the vertical  lines. The exceptional line $E_1$ is invariant and the unique singularity along $E_1$ is  a 
saddle-node  with Milnor number $\mu=4$, weak separatrix $E_1$ and strong separatrix $\overline{L}$. This curve is the unique invariant fiber whose corresponding multiplicity is $4$.  The foliation on $\mathbb{F}_1$  is already a nef  model.  The  cotangent line bundle can be computed as above,
\[  T^*\calf_{nef}  =  \calo(-2 \overline{L}) \otimes \calo(  4  \overline{L} ) =  \calo(2 \overline{L})\quad\mbox{and}\quad    \kappa(\calf) = 1.\]   
\end{exa}

\begin{exa}\label{novoexemplograu4}
Consider the Shamsuddin type foliation of degree $4$ in the plane associated to 
\[\omega = ( (x^3 + 1) y +  5 x^4  - x^3  - 2 x^2 + 4 x      ) dx - dy =0.\]
At   $(1:-5:0)$ there is  a saddle-node with Milnor number $\mu=5$. 
At $(0:1:0)$  there is a quartic singular point (algebraic multiplicity $=4$). The foliation obtained after 
 blowing up this point is Riccati, having  just one singular point along $E_1$ which is a saddle-node, with Milnor number $\mu=5$, weak separatrix $E_1$ 
 and strong separatrix $\overline{L}$.

 The foliation on $\mathbb{F}_1$  is already a nef  model. The multiplicity  of  $\overline{L}$ as an invariant fiber is $5$  and as before we obtain
\[  T^*\calf_{nef}  =  \calo(-2 \overline{L}) \otimes \calo(5  \overline{L}) =  \calo(3 \overline{L})\quad\mbox{and}\quad  \kappa(\calf) =1.\]   
\end{exa}

\begin{exa}\label{shamgrau8}
 From \cite[Exa. 13.3.7]{NoT} we have  a foliation of degree $8$  of  Shamsuddin type given by
 \[ \omega = ( (x^3 +1) y + x^8 + 3 x^5 + 1)  dx - dy = 0;\]
on $\C^2$ it has neither algebraic invariant curve nor singularities. At $(0:1:0)$  there is a highly degenerate singularity (with algebraic multiplicity $=8$).
The blowing-up at this point produces a  Riccati foliation, which is not reduced yet. It needs  four additional  blowing-ups. From  the second blown up point  to the fifth, the algebraic multiplicity is  $=2$. 
Along the fifth  exceptional line $E_5$  there are  three singular points: two saddle-nodes  with Milnor number $\mu=5$  and one Morse point. 
The  foliation obtained is  reduced  but not a relatively minimal  model. 

To obtain a relatively  minimal  model we contract $\overline{L}$, $E_2$, $E_3$ and $E_4$, in this order. 
\[\includegraphics{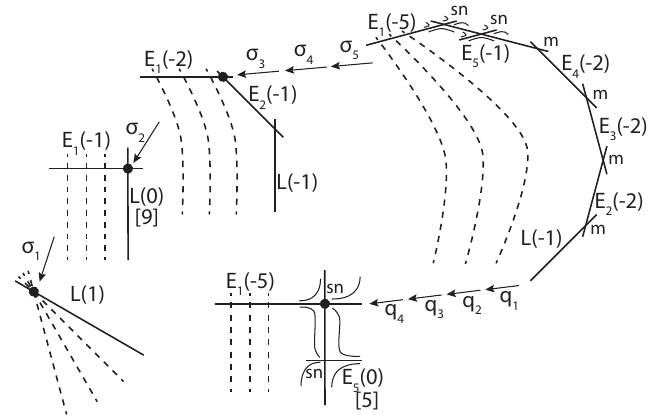}\] 
As we contract foliated  exceptional curves with Morse points, the strict transform of $E_5$  contains just two saddle-nodes. In this example 
the relatively minimal model is already  a  nef model (on the Hirzebruch surface   $\mathbb{F}_5$).  

Arguing as in the precedent example we deduce
\[  T^*\calf_{nef}  =  \calo(-2  \overline{E}_5 )  \otimes \calo(  5  \overline{E}_5) =  \calo(3 \overline{E}_5) \quad\mbox{and}\quad   \kappa(\calf) = 1.\] 
\end{exa}

\begin{exa}\label{nilp} This example is not of Shamsuddin type but is of Riccati type. From \cite{Moulin} (see also \cite{CouQuad}) we consider the foliation of degree  $ 2 $  in the projective plane  associated to  
 \[ \omega =  dx - (x^2 + y)  dy   =   0.\]
 At the horizontal  infinity point  there  is a radial point (Milnor number $\mu=1$), whose  
 blowing-up  produces a Riccati foliation completely  transverse to the exceptional line $E_1$.
 There
 is   a nilpotent
 singularity  at the vertical  infinity point whose Milnor number is $\mu=6$ (thanks to Darboux's formula in the plane, \textit{cf.} \cite[p. 19]{BrBG}).
 The blowing-up at the nilpotent singular point produces an invariant exceptional curve $E_2$ having  just one quadratic  singularity, at the intersection 
 with $\overline{L}$. The blowing-up at this quadratic singularity  produces  three singularities  along $E_3$: 
 two of them being non-degenerate reduced singularities, placed at the intersections of $E_3$  with $\overline{L}$ and $E_2$,  and a third one being a saddle-node, with strong separatrix $E_3$. We assert that the  Milnor number of 
this saddle-node is $\mu =4$: indeed, it follows from the diagram on the top of \cite[p. 47]{BrBG} and the fact that a nilpotent singular point in $\P^2$ has Milnor number $\mu=6$.
 
The foliation obtained is not a nef model for the original Riccati foliation; denote as $\pi$ the corresponding adapted fibration. The  morphism  $\rho = \rho_2\circ \rho_1$ contracts   two $(-2)$-curves and produces  
a  singular surface 
with two quotient  singularities  $q_1,q_2$  along  the strict  transform  of $E_3$ (where there is also a   saddle-node    with Milnor number $\mu= 4$). The resulting foliation of this surface is nef.
\[\includegraphics{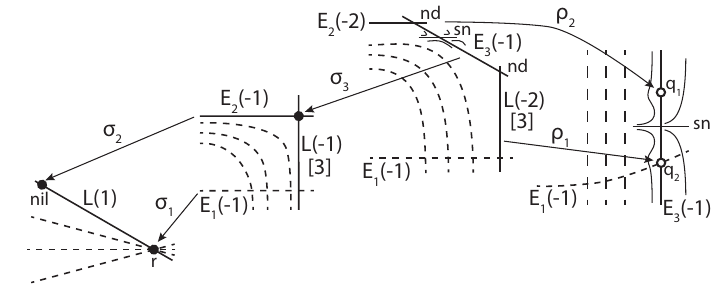}\] 
We deduce (see \cite[p. 20]{BrPisa}) 
 \[ deg(\pi_*(T^*{\calf_{nef}})) =  -2 + \frac{4 + 1}{2}  =  \frac{1}{2}    > 0 \quad  \mbox{and}\quad  \kappa ( \calf ) = 1,   \]
where  the leftmost term is the degree of the rational divisor $\pi_*(T^*{\calf})\in \mbox{Pic}(\P^1)\otimes \mathbb{Q}$.
\end{exa}

All  the remaining   examples  have $\kappa (\calf) =2$.

\begin{exa}\label{Cerveauex}
Consider the foliation $\calf$ of degree $\deg(\calf )= 2$   in the projective plane  associated to the following  equation in $\C^2$:
\[ \omega =  x \,(1+xy)\,   dx - (1+xy+x^3)\,   dy  =  0 \]
which was taken from \cite[Prop. $1.3$]{Ce}. The line at infinity  is  \emph{not} invariant and  there is  just one 
saddle-node  with  Milnor  number $\mu=7$ at infinity. So the singularity of $\calf$ is reduced. The cotangent
line bundle is $T^*{\calf} = \calo (1)$ and $\kappa (\calf) = 2$. As there is no curve with negative self intersection in $\Pd$, $\calf$ is its own nef model.
The group   $Pol(\calf|_{\C^2})$  contains the linear automorphism $L_j(x,y)\mapsto (j \cdot   x , j^2 \cdot   y)$
where $j$ is 
a primitive cubic root of the unity and 
$L_j^*( \omega)  =  j^2  \cdot  \omega.$
The  affine line  $x=0$  is  transverse   to the   foliation and, for $c\neq 0$,  the affine  line $x= c$ has  one movable  tangency.
\end{exa}

\begin{exa}\label{Nofamilia}

From  \cite{No}, we obtain  a   family  of foliations $\calf_{k}$ of degree $k\geq 2$  in the plane associated to simple derivations. It is defined by the family of $1$-forms
\[\omega_{k} =  (y^k + x) dx  - dy, ~~~k \in \N^*.\]
For $k=2$ this coincides with Example~$\ref{nilp}$,  up  to permutation of $(x,y)$.
     
We assert that  $\kappa (\calf_k) =2$ for all $k \geq  3$. At the vertical infinity point, each  $\calf_{k}$
has  a radial point $p$.    The exceptional line of the blowing-up at $p$  
belongs to the contact divisor
between  the transformed foliation and the rational fibration corresponding to the pencil of lines by $p$.  
For simplicity, let's focus  on  the case $k=3$. The reduction of singularities  of $\calf_3$ is made up of $4$ blowing-ups 
at quadratic singularities of the foliation. The fourth blowing-up introduces $E_4$ having  a saddle-node 
and $2$ extra non-degenerated reduced points (at the  intersections of  $E_4$ with 
the strict transforms of $E_2$ and $E_3$). 
The  Zariski decomposition  of 
the cotangent line bundle is
\[T^*{\overline{\mathcal{F}_3}}  \equivq     {\bf P}   +  \frac{2}{3} \overline{L}  + \frac{1}{3} \overline{E_3} +  \frac{1}{3} \overline{E}_2.  \]
The nef model  is obtained after  contraction 
of the support  of ${\bf N}$  and introduces  two  quotient singularities  of the surface $q_1,q_2$.
\[\includegraphics{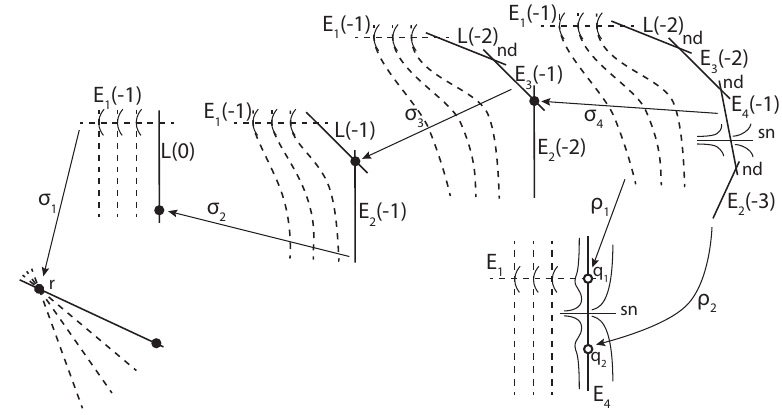}\]

Since the   $4$  blown up singularities were  quadratic ($l(\calf)= 2$),  we can compute   
\[ T^*{\overline{\mathcal{F}_3}} \cdot  T^*{\overline{\mathcal{F}_3}}  =  (deg(\mathcal{F}_3)-1)^2 - 
\sum_{i=1}^4  (l_{p_i}(\mathcal{F}_3) -1)^2 = 4 - 4 = 0\]
Combining this with 
\[{\bf N}\cdot {\bf N} =  ( \frac{2}{3} \overline{L}  + \frac{1}{3} \overline{E_3} +  \frac{1}{3} \overline{E}_2)^2 = -1  \]
and   ${\bf P}\cdot {\bf N} = 0$,   the conclusion is ${\bf P}\cdot {\bf P} = 1$. Therefore the  numerical  
Kodaira dimension is $2$  and  also    $\kappa (\calf_3) =2$. 
\end{exa}

\begin{exa}\label{jordan}
From  \cite{J} we have a  foliation $\calf$ of  degree $3$ in the projective plane associated to 
\[  \omega = y^3 \,  dy  -  (1-xy) \,  dx.\]
 At $(0:1:0)$, there is  a quadratic 
 singularity whose reduction is composed  by $3$  blow ups. The foliation has  $\kappa (\calf) =2$  and its
 nef model
 is  not much different  than the one  of  Example  \ref{Nofamilia}. At last, we remark that   
 the  affine lines  $y=c$, $c\neq 0$, exhibit one movable tangency point  with the  foliation (at the intersection of $y=c$ and  
 $y = \frac{1}{x} $). The  affine line $y=0$  is   completely  transverse to the foliation, a property that  will  be useful
 in   Section \ref{recobrimentos}. 
 
\end{exa}

\begin{exa}\label{surj}
From    \cite{SK}   we have examples of  foliations  $\calf_{r,s,g}$  in the projective plane with degrees    $\deg(\calf_{r,s,g}) = s + 1$, defined  for $r,s \in \N$ such that $r+2\leq s$ and $g\in \C^*$ by  
\[  \omega_{r,s,g} :=  (x y^s + g  )\,  dx - y^r \, dy. \]
We have a singular  point at $(0:1:0)$ with algebraic multiplicity $s$. At  $(0:0:1)$, there is  a singular  point with algebraic multiplicity $ 3$.   Except for $y=0$, all  horizontal affine 
lines  exhibit   one  movable   contact point  with the restricted  foliation ${\calf}_{r,s,g}|_{\C^2}$. 
However, the affine  line  $y=0$ is completely  transverse to the foliation. 

\end{exa}

\begin{exa}\label{odani}
 
According to \cite{Odani}, the  foliations in the plane  associated to the Li\'enard equations  
\[ \omega  :=   (f(x) \cdot y + g(x) )\,  dx + y \, dy = 0,\quad f, g \in \C[x]    \]
do \emph{not}  have  algebraic solutions in $\C^2$  if three  conditions are satisfied:   $i)$  $f,g \neq 0$, $ii)$ $\deg(f) \geq \deg(g)$ and 
$iii)$ $ \frac{f}{g}$ is not constant.
These  foliations do not have  singularities  in $\C^2$   exactly  when    $g(x) = c \in \C^*$. 
Therefore, to produce  foliations 
associated  to simple derivations   it suffices  to take a constant  $g$  and a non-constant polynomial $f(x)$.  In this case, the affine line $y=0$ 
is everywhere transverse to the foliation while the horizontal line $y=c$   exhibit $\deg(f)$ tangencies   with the foliation. 
The  line  at infinity is invariant by the  extended foliation.  
\end{exa}

Now we establish some birational (non)-equivalences among the Examples.
\begin{pro}\label{classes}
\leavevmode
\vspace{-0,4cm}
\begin{enumerate}[$($i$)$] 
\item \label{equiv i}Example~$\ref{GL}$ is isomorphic by a linear transformation to  a foliation of Shamsuddin type.
\item \label{equiv ii}Examples~$\ref{novoexemplograu4}$ and $\ref{shamgrau8}$  are  equivalent by a polynomial  automorphism  of degree five.
\item \label{equiv iii}The foliations of  Example~$\ref{shamquadratico}$ and  Example~$\ref{shamgrau3}$   are  \emph{not}  birationally equivalent.
\end{enumerate}
\end{pro}

\begin{proof}
  ~\\
 $(\ref{equiv i})$.
With the  linear  change of variables $y = u - 2 v$, $   x = v$,    from the equation of   Example  \ref{GL} we   obtain the Shamsuddin type
\[ \eta = 2   dv - ( 1 + v u)  du\]
\noindent  $(\ref{equiv ii})$.We start    with   the  nef model  of  Example  $\ref{shamgrau8}$  in the  Hirzebruch surface $\mathbb{F}_5$.
After an elementary transformation we pass  to $\mathbb{F}_4$ keeping the saddle-nodes  and the  multiplicity $[5]$  of the unique invariant fiber. This is shown in the next figure.
\[\includegraphics{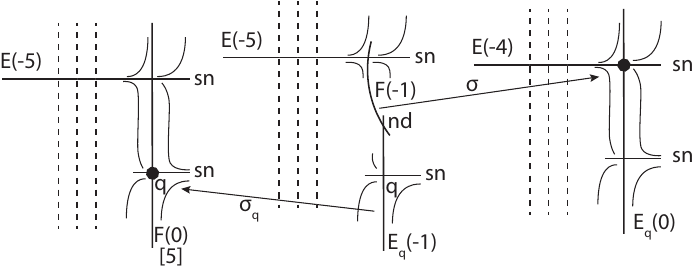}\] 

After this we  perform   three more    elementary  transformations  
to obtain   a foliation  of   $\mathbb{F}_1$ 
and then, after 
contraction of the  $(-1)$-section, we arrive at the  projective  plane.  The  pair  of saddle-nodes  on the  invariant vertical   fiber of $\mathbb{F}_5$
is transformed into a    pair of singularities along an invariant projective line  of a degree   $4$  foliation.   The net effect in the plane  can be described  concretely    
by means  of a polynomial  isomorphism  of  $\C^2$. Indeed, 
\[ \eta  :=  ( (x^3 + 1) y +  5 x^4  - x^3  - 2 x^2 + 4 x      ) dx - dy      \]
defines the foliation of Example  \ref{novoexemplograu4} and, if  $R: \C^2\to \C^2$ is  
\[(x,y) \mapsto  ( x\, , \, y + x^5 + 2 x^2 + 1)\]                                               
then  $R^*(\eta)$ is  the $1$-form of  Example  \ref{shamgrau8}.

 \noindent$(\ref{equiv iii})$.
We use   the  birational  invariant  $g(\calf)$ defined in \cite[p. 139]{Me}.  
In  Example~$\ref{shamquadratico}$ and  Example~$\ref{shamgrau3}$ the  invariants are   $g(\calf) = 2$ and  $g(\calf)= 3$, respectively.  
\end{proof}
  
Proposition~$\ref{classes}$-$(\ref{equiv ii})$  illustrates   the general   issue of finding the simplest (of least degree) plane birational model for foliations of the plane or
derivations.  Algorithmic procedures toward 
this objective would be of great utility.

 \section{Proof of Theorem~$A$}
 In order to prove Theorem A we proceed by contradiction: suppose $Bir(\F)$ is infinite.
From our assumption on algebraic invariant curves, $\F$ cannot have a rational first integral. We have two cases: either
\begin{enumerate}
\item  \label{cas a)} there exists a birational model $(X,\tilde{\F})$ of $(\Pd,\F)$ such that $\Aut(\tilde{\F})=Bir(\tilde{\F})$ or
\item \label{cas b)} for every birational model $(X,\tilde{\F})$, the inclusion $\Aut(\tilde{\F})\subset Bir(\tilde{\F})$ is strict.
\end{enumerate}

In case $(\ref{cas a)})$, we may apply \cite[Thm. $1.1$]{CF} to $(X,\tilde{\F})$, as fibrations are excluded we have one of the following situations.
\begin{enumerate}
\item[$(i)$] There exists a non trivial holomorphic vector field $\der$ on $\tilde{X}$ defining a one parameter subgroup of $\Aut(\tilde{\F})$ or
\item[$(ii)$] the surface $X$ is a generalized Kummer surface (see \cite[Exa. $1.1$]{CF}), $X$ is a quotient of an abelian surface $A$ and $\tilde{\F}$ lifts to $A$ as a linear foliation $\G$.
\end{enumerate}
In case $(\ref{cas b)})$, we may apply \cite[Thm. $1.2$]{CF}. Since fibrations are excluded, we are in the situation of   \cite[Exa. $1.3$]{CF}, in particular up to passing to a birational model $(X,\tilde{\F})$,
\begin{enumerate}
\item[$(iii)$] the surface $X$ is a  finite quotient of $S=\Pu\times \Pu$ and $\tilde{\F}$ lifts to a foliation $\G$ of $S$ given by a differential form $\alpha wdz+\beta zdw,$ for some $\alpha,\beta\in \C$.
\end{enumerate}

In cases $(ii)$ and $(iii)$ the foliation $\G$ has Kodaira dimension $0$. By the remark in \cite[p. $20$]{BrBG} or the proof of \cite[Lemme $3.2.8$]{Tou}, this forces $\kappa(\tilde{\F})\leq0$ and contradicts Theorem \ref{kodsimples}.

So we only need to derive a contradiction from situation $(i)$ to complete the proof. Moreover, if $\der$ is tangent to $\tilde{\F}$, we may use \cite[Prop. $6.6$ $(iii)$]{BrBG} to see that $\tilde{\F}$ is a Riccati foliation with two distinct adapted fibrations and consequently $\kappa(\F)=0$, contradicting  Theorem~\ref{kodsimples}. 

The conclusion is given by the following argument, proposed by Jorge Pereira. If $\der$ is not tangent to $\tilde{\F}$ consider its image $\tilde{\der}$ in $\Pd$ and take a rational $1$-form $\omega$ defining $\F$.
The form $\eta=\frac{\omega}{\omega(\tilde{\der})}$ is closed, by the computation \cite[Proof of Cor. $2$]{PS} inherited from \cite{CeMa}. The poles of $\eta$ give $\F$-invariant algebraic curves, so that $\eta$ has no poles in $\C^2$. Subsequently, the first integral $\int_\star^x \eta$ has no monodromy and gives a rational first integral for $\F$, contradiction.
\qed

\section{Polynomial  symmetries  of foliations   associated to simple derivations}\label{recobrimentos}

Recall that for a foliation $\mathcal{H}$ of $\C^2$, the group $Pol(\mathcal{H})$ is defined as the group of polynomial automorphisms of $\C^2$ that preserve $\mathcal{H}$.
If $\omega$ is a polynomial $1$-form with isolated zeroes that defines $\mathcal{H}$ and $\phi$ is a polynomial automorphism of $\C^2$, then $\phi\in Pol(\mathcal{H})$ if and only if there exists $c\in \C^*$ such that $\phi^*\omega=c\omega$, see Remark~$\ref{constant factor}$.

\begin{pro}\label{pullbackfol}
 
Given     $n\geq 2$ and   $B>0$, there exists a foliation  $\calg$  
 associated to a simple derivation and an element in  $Pol(\calg|_{\C^2})$   of order  $n$ and degree greater than $B$.
\end{pro}

\begin{proof}
Let $\calf$ be a foliation of the projective plane $\P ^2 = \C^2 \cup L_{\infty}$ with the following two properties: 

\begin{enumerate}[$(a)$]
\item \label{sym a} $\calf$ is associated  to  a simple derivation,
\item \label{sym b} there exists   some  affine  straight  line      completely transversal to the foliation $\calf|_{\C^2}$.
\end{enumerate}
Up to an affine transformation, we can suppose that the line in $(\ref{sym b})$  is $x=0$. Consider the $n$ to $1$  rational   map  $\phi_n: \C^2 \to \C^2$ given by $\phi_n(x,y) = (x^n,y)$. 
The foliation $\calf$ is defined by a polynomial $1$-form $\omega = a(x,y)dx +  b(x,y)dy$.  

The map $\phi_n$ extends to a birational map $\widehat{\phi}_n: \Pd\dasharrow \Pd$ and let  $\calg_0=\widehat{\phi}_n^*\F$. 
It is  defined by the 1-form $\phi_n^*(\omega) = a(x^n,y) d(x^n) + b(x^n,y) dy.$

Since the affine line  $x=0$  is supposed to be completely  transverse to $\calf$, then $\calg$  has no singularity along $x=0$. Note also that in $\C^2\setminus \{x=0\}$
the map $\phi_n(x,y) = (x^n , y)$  is a local isomorphism, so the pull-back  does  not introduce any singularity. Any algebraic ${\G_0}\vert_{\C^2}$-invariant curve would descend to an ${\F}\vert_{\C^2}$-invariant curve, so that no such curve exists. Hence  $\calg_0$  is associated to a  simple derivation.

On the other hand, let $\xi$ be a primitive $n$-th root of unity. The linear automorphism $ T_{\xi}( x, y) =  (\xi \cdot x , y )$ preserves $\calg_0$ and has order $n$. 
Note that if $\der_n$ is the derivation dual to the form  $\phi_n^*(\omega)$, then $T_{\xi} \not\in \Aut(\der_n)$  because $T_{\xi}^*(\phi_n^*(\omega))= \phi_n^*(\omega)$  and   $Jac(T_{\xi}) = \xi\neq 1$.

Now, for any polynomial $\tau\in \C[x]$, consider the automorphism of $\C^2$ defined by \[P_{\tau}  (x,y) = (x  \, , \,  y +\tau(x))\] 
and set $\calg:=P_{\tau}^*\calg_0$. If  the polynomial $\tau(\xi\cdot x)-\tau(x)$ has degree $d\geq 1$, (\textit{e.g.} if $\deg_x(\tau)=d$ and $(d,n)=1$) the map 
 \[\Gamma_{\xi,\tau}(x,y):=P_{\tau}^{-1}\circ T_{\xi}\circ P_{\tau}(x,y)=   ( \xi \cdot x \, ,\,   y + \tau(\xi\cdot x) -  \tau(x))\]
defines a polynomial automorphism of degree~$d$ and order $n$ in $Pol(\calg_0)$, which completes the proof.
\end{proof}

The properties $\ref{sym a})$ and  $\ref{sym b})$ used in  the proof of Proposition \ref{pullbackfol} are verified  in all examples of  Section~$\ref{Exemplos}$,  so that we have plenty of examples. 

For instance,  start with   Example \ref{shamquadratico}  induced by  $\omega  =  (x y + 1) dx - dy$.
Assume $n=2$, so $\phi_2(x,y) = (x^2, y )$. We obtain a foliation $\calg_0$ defined by the $1$-form \[\phi_2^*(\omega)  = (2 x^3 y + 2 x )  \,  dx - dy\] and  $Pol(\calg_0)$ contains the linear automorphism $T(x,y)= (-x,y)$. 
Now, we consider the automorphism $P_{\tau} (x,y) = ( x , y +x^3)$ and define 
\[\Omega:= P_{\tau}^*(   \phi_2^*(\omega)    )  =   (2 x^3 y-2 x^6+3 x^2+ 2 x) dx - dy.\]
Then the involution  $\Gamma_{\xi , \tau} = \Gamma_{-1, x^3}:  (x , y )  \mapsto ( - x ,   y - 2 x^3                    )$
preserves  the $1$-form $\Omega$.

At last,  note that the foliations that we have constructed in the proof of Proposition~$\ref{pullbackfol}$ do not have minimal degree  in their birational
classes,
due to the fact that the automorphism $P_{\tau}$ has degree greater that one. In other terms, such foliations  are not \emph{primitive} in the sense  of \cite{CD}. 
On the other side, the foliation  of   Example~$\ref{Cerveauex}$   is primitive  but the exhibited automorphism is linear.  

This raises the following question: \emph{Are there non-linear polynomial  automorphisms of primitive  foliations associated to simple derivations ?}

By \cite[Cor. p $293$]{MR1708643}, such automorphisms would necessarily be  conjugated to automorphisms of the form $(x,y)\mapsto(ax+P(y),by+c)$ with $a,b\in \C^*$, $c \in \C$ and $P\in \C[y]$. 

\appendix
\section{}\label{app}
This appendix explains and corrects a slight error in the proof of \cite[Prop. $6.6$]{BrBG}, case of rational surfaces. 

 Suppose $p: \mathbb{F}_n\to\Pu, n>0$, is the fibration  of a Hirzebruch surface and one has a Riccati foliation  induced by a global holomorphic vector field $v$ on $\mathbb{F}_n$. One can cover $\Pu$ by two open subsets $U_0$ and $U_{\infty}$, both isomorphic to $\C$, in such  a way that one has trivializations
$p\vert_{U_i}\simeq U_i\times \Pu\to U_i$. One can describe $\mathbb{F}_n$ by the gluing of the two products $U_i\times \Pu$ by the rule $(w,z)\sim(t,y)\Leftrightarrow [wt\neq0\mbox{ and } w=1/t \mbox{ and } z=t^ny]$ where $z$ and $y$ are affine charts of $\Pu$ and $w$ and $t$ are coordinates on $U_0$ and $U_{\infty}$, respectively. The projection $p$ then corresponds to the first projections of these products and the negative section $\Gamma$ is  the closure of $z=0$.
The holomorphic vector field $v_0=v\vert_{p^{-1}(U_0)}$ takes the form $(\star)$ of Section~$\ref{localform}$. The section $\Gamma$ is rigid and hence invariant by $v$, so that $a=0$.  
The transform $v_{\infty}$ of $v_0$ in the second chart must have the same form. A straightforward calculation then yields the following necessary and sufficient condition for the holomorphicity of $v_\infty$ on $U_\infty\times \Pu$, where $\C_k[w]$ are the polynomials of degree~$\leq k$.
\[\left \lbrace\begin{array}{l}
d\in \C_2[w], d=d_0+d_1w+d_2w^2,\\
b=b_0-nd_2w,\\
c\in \C_n[w].
\end{array}
\right.\]
So that $v_0$ takes the form
$v_0=d(w)\partial_w+z(b_0-nd_2w+zc(w))\partial_z$.
 The restriction $v\vert_{\Gamma}$ is holomorphic, hence, choosing well the coordinates $w,t$ on $\Pu$ in the beginning, one may suppose $d(w)=w^2$ or $d(w)=\lambda w$, $\lambda\in \C^*$ or $d=0$.  
 The option $d=0$ is excluded, since we have a Riccati foliation with respect to $p$.

If $b_0 \neq 0$ the foliation has two distinct singularities in the fiber $w=0$.
Otherwise $b_0=0$ and $v_0=d(w)\partial_w+z(-nd_2w+zc(w))\partial_z$. In these conditions the vector field vanishes on $w=0$ or $c_0:=c(0)\neq0$.

If $b_0=0$, $c_0\neq 0$ and $d=\lambda w$ then $d_2=0$ and
$v_{\infty}=-\lambda t\partial_t+y(n+yk(t))\partial_y$ where $k(t)=t^nc(1/t)$ has degree $n$. One sees that $v$ vanishes in two distinct points of the fiber $t=0$. 

If $b_0=0$, $c_0\neq 0$ and $d=w^2$, then $v$ vanishes only at $(w,z)=(0,0)$ in $\mathbb{F}_n$ and this contradicts Brunella's assertion concerning existence of a  zero of $v$ outside $\Gamma$ in \cite[p. 78]{BrBG}.

In conclusion, the argument of  \cite[p. 78]{BrBG} works fine except if the situation $b_0=0$, $c_0\neq 0$ and $d=w^2$ appears in the sequence of elementary transformations.
This corresponds to the vector fields of the form $v_0=w^2 \partial_w+z(-nw+zc(w))\partial_z$. 

However, in this case, one can describe an additional branch of the algorithm that allows to reach a holomorphic vector field on $\Pu\times\Pu$.
Perform an elementary transformation centered at the unique singularity of the vector field. From the foliation perspective, this yields a simple pole for the Riccati equation, with trivial monodromy.
 In these conditions, there is a sequence of elementary transformations centered at singularities of the foliation that eliminates the pole, the pole staying of order at most $1$ after each step. On the other hand, the restriction of the vector field to the negative section is unaltered outside $w=0$, so that its vanishing order at this point remains $2$ and that the vector field vanishes on the whole fiber once the Riccati foliation has a simple pole.
 One concludes that the proposed sequence of elementary transformations preserves the holomorphicity of the vector field. The disappearance of the foliation's singularities shows the obtained bundle is $\Pu\times\Pu\to \Pu$, due to Camacho-Sad formula.

\begin{rem}\label{second sing} In the above discussion, if $d(w)=\lambda w$, $\lambda \in \C^*$, and the Riccati foliation has a unique invariant fiber, then this fiber corresponds to a simple pole and has trivial monodromy so that it can be eliminated by a sequence of elementary transformations centered at singularities of the foliation. We conclude that in the presence of a unique invariant fiber, one always passes from a ``Riccati'' holomorphic vector field on $\mathbb{F}_n$ to a  holomorphic vector field on $\Pu\times \Pu$ by a sequence of elementary transformations centered at singularities of the foliation.
\end{rem}

\newcommand{\etalchar}[1]{$^{#1}$}
\providecommand{\bysame}{\leavevmode\hbox to3em{\hrulefill}\thinspace}
\providecommand{\MR}{\relax\ifhmode\unskip\space\fi MR }
\providecommand{\MRhref}[2]{%
  \href{http://www.ams.org/mathscinet-getitem?mr=#1}{#2}
}
\providecommand{\href}[2]{#2}

\end{document}